\documentclass[reqno,a4paper, 11pt]{amsart}

\usepackage[a4paper=true,pdfpagelabels]{hyperref}
\usepackage{graphicx}

\usepackage[ansinew]{inputenc}
\usepackage{amsfonts,epsfig}
\usepackage{latexsym}
\usepackage{amsmath}
\usepackage{amssymb}
\usepackage{mathabx}

\newtheorem{theorem}{Theorem}
\newtheorem{lemma}[theorem]{Lemma}

\newtheorem{proposition}[theorem]{Proposition}

\newtheorem{lettertheorem}{Theorem}
\newtheorem{letterlemma}[lettertheorem]{Lemma}

\theoremstyle{definition}

\theoremstyle{remark}

\numberwithin{equation}{section}

\newcommand{\D}{\mathbb{D}}
\newcommand{\DD}{\widehat{\mathcal{D}}}

\newcommand{\Dd}{\widecheck{\mathcal{D}}} 
\newcommand{\DDD}{\mathcal{D}}

\newcommand{\N}{\mathbb{N}}

\newcommand{\R}{\mathcal{R}}

\newcommand{\C}{\mathbb{C}}

\renewcommand{\phi}{\varphi}

       \def\b{\beta}        \def\g{\gamma}
           
     \def\om{\omega}      
\def\s{\sigma}       \def\t{\theta}       
                  \def\z{\zeta}

\usepackage{verbatim}

\begin{document}

\title[Radial two weight inequality for maximal Bergman projection]{Radial two weight inequality for maximal  Bergman projection induced by a regular weight}


\thanks{This research was supported in part by Ministerio de Econom\'{\i}a y Competitivivad, Spain, projects
MTM2014-52865-P and MTM2015-69323-REDT; La Junta de Andaluc{\'i}a,
project FQM210; Academy of Finland project no. 268009.}

\date{\today}

\begin{abstract}
It is shown  in quantitative terms that the maximal Bergman projection
    $$
    P^+_\om(f)(z)=\int_\D f(\zeta)|B^\om_z(\zeta)|\om(\zeta)\,dA(\zeta),
    $$
is bounded from $L^p_\nu$ to $L^p_\eta$ if and only if
    \begin{equation*}\
    \sup_{0<r<1}\left(\int_0^r\frac{\eta(s)}{\left(\int_{s}^1\om(t)\,dt\right)^p}\,ds\right)^{\frac{1}{p}}
    \left(\int_r^1\left(\frac{\om(s)}{\nu(s)^\frac{1}{p}}\right)^{p'}ds\right)^{\frac{1}{p'}}<\infty,
    \end{equation*}
provided $\om,\nu,\eta$ are radial regular weights. A radial weight $\sigma$ is regular if it satisfies $\s(r)\asymp\int_{r}^1\s(t)\,dt/(1-r)$ for all $0\le r<1$.
It is also shown that under an  appropriate additional hypothesis involving $\om$ and $\eta$, the Bergman projection $P_\om$ and $P^+_\om$ are simultaneously bounded.
\end{abstract}

\author{Taneli Korhonen}
\address{University of Eastern Finland, P.O.Box 111, 80101 Joensuu, Finland}
\email{taneli.korhonen@uef.fi}

\author{Jos\'e \'Angel Pel\'aez}
\address{Departamento de An\'alisis Matem\'atico, Universidad de M\'alaga, Campus de Teatinos,   29071 M\'alaga, Spain}
\email{japelaez@uma.es}

\author{Jouni R\"atty\"a}
\address{University of Eastern Finland, P.O.Box 111, 80101 Joensuu, Finland}
\email{jouni.rattya@uef.fi}

\maketitle

\section{Introduction and main results}

A function $\om:\D\to[0,\infty)$, integrable over the unit disc $\D$, is called a weight. It  is radial if $\om(z) = \om(|z|)$ for all $z\in\D$. For $0<p<\infty$ and a weight $\om$, the Lebesgue space $L^p_\om$ consists of complex-valued measurable functions $f$ in $\D$ such that
    $$
    \|f\|_{L^p_\om}=\left(\int_\D|f(z)|^p\om(z)\,dA(z)\right)^{\frac1p}<\infty,
    $$
where $dA(z) = \frac{dx\,dy}\pi$ denotes the element of the normalized Lebesgue area measure on $\D$. The weighted Bergman space $A^p_\om$ is the space of analytic functions in $L^p_\om$. If the norm convergence in the Hilbert space $A^2_\om$ implies the uniform convergence on compact subsets of $\D$, the point evaluations are bounded linear functionals on $A^2_\om$. Therefore there exist reproducing Bergman kernels $B^\om_z\in A^2_\om$ such that
    $$
    f(z)=\langle f,B^\om_z \rangle_{A^2_\om}
    = \int_\D f(\zeta)\overline{B^\om_z(\zeta)}\om(\zeta)\,dA(\zeta),\quad z\in\D, \quad f\in A^2_\om.
    $$
The Hilbert space $A^2_\om$ is a closed subspace of $L^2_\om$, and hence the orthogonal projection from $L^2_\om$ to $A^2_\om$ is given by
    $$
    P_\om(f)(z) = \int_\D f(\zeta)\overline{B^\om_z(\zeta)}\om(\zeta)\,dA(\zeta),\quad z\in\D.
    $$
The operator $P_\om$ is the Bergman projection.

In this paper we will characterize the radial two-weight inequality
    \begin{equation}\label{eq:twoweight}
    \|P^+_\om(f)\|_{L^p_\eta}\le C \|f\|_{L^p_\nu},\quad f\in L^p_\nu,
    \end{equation}
for the maximal Bergman projection
    $
    P^+_\om(f)(z)=\int_\D f(\zeta)|B^\om_z(\zeta)|\om(\zeta)\,dA(\zeta)
    $
under certain smoothness requirements on the three radial weights involved. The question of when \eqref{eq:twoweight} is satisfied is an open problem even in the very particular case $\om=\nu=\eta$ if no preliminary hypotheses is imposed on the radial weight.

Two weight inequalities for classical operators have attracted a considerable amount of attention in Complex and Harmonic Analysis, and are closely connected to other interesting questions in the area \cite{AlPoRe,HyAn12,Laceyetal14,LaceyWicketal18,PelRatKernels,PRW}. The most commonly known result on Bergman projection is due to Bekoll\'e and Bonami~\cite{B1981,BB}, and concerns the case when $\nu=\eta$ is an arbitrary weight and the inducing weight $\om$ is standard, that is, of the form $\om(z) = (1-|z|^2)^\alpha$ for some $\alpha>-1$; see \cite{ACJFA12,PRW,PottRegueraJFA13} for recent extensions of this result. In this classical case, the Bergman reproducing kernel $B^\om_z(\z)$ is given by the neat formula $(1-\overline{z}\zeta)^{-(2+\alpha)}$. However, for a general radial weight $\om$ such explicit formulas for the kernels do not necessarily exist, and that is one of the main obstacles in tackling~\eqref{eq:twoweight}. Moreover, kernels induced by radial weights may have zeros, and that of course does not make things any easier. Nonetheless, \eqref{eq:twoweight} has been recently characterized in the particular case $\nu=\eta$ provided $\om$ and $\nu$ are regular weights~\cite{PelRatKernels}.

For a radial weight $\om$, we assume throughout the paper that $\widehat{\om}(z) = \int_{|z|}^1 \om(s)\,ds > 0$ for all $z \in \D$, for otherwise the Bergman space $A^p_\om$ would contain all analytic functions in $\D$. A radial weight $\om$ belongs to the class~$\DD$ if there exists a constant $C=C(\om)>1$ such that $\widehat{\om}(r)\le C\widehat{\om}(\frac{1+r}{2})$ for all $0\le r<1$. Moreover, if there exist $K=K(\om)>1$ and $C=C(\om)>1$ such that
    \begin{equation}\label{K}
    \widehat{\om}(r)\ge C\widehat{\om}\left(1-\frac{1-r}{K}\right),\quad 0\le r<1,
    \end{equation}
then we write $\om\in\Dd$. The intersection $\DD\cap\Dd$ is denoted by $\DDD$.
A radial weight $\om$ is regular if $\widehat{\om}(r)\asymp\om(r)(1-r)$ for all $0\le r<1$. The class of regular weights is denoted by $\R$,
and $\R\subsetneq\DDD$. For basic properties of these classes of weights and more, see~\cite{PelSum14,PR2014Memoirs,PR2018} and the references therein.

The main result of this study is the following theorem, which provides a quantitative description of the boundedness of $P^+_\om: L^p_\nu\to L^p_\eta$   in terms of a Muckenhoupt-type condition related to  weighted Hardy operators.

\begin{theorem}\label{th:P+threeweights}
Let $1<p<\infty$, $\om,\nu\in\R$ and $\eta\in\Dd$. Then $P^+_\om: L^p_\nu\to L^p_\eta$ is bounded if and only if
    \begin{equation}\label{eq:necessity3}
    M_p(\om,\nu,\eta)=\sup_{0<r<1}\left(\int_0^r\frac{\eta(s)}{\widehat{\om}(s)^p}\,ds\right)^{\frac{1}{p}}
    \left(\int_r^1\left(\frac{\om(s)}{\nu(s)^\frac{1}{p}}\right)^{p'}ds\right)^{\frac{1}{p'}}<\infty.
    \end{equation}
Moreover, $\|P^+_\om\|_{L^p_\nu\to L^p_\eta}\asymp M_p(\om,\nu,\eta)$.
\end{theorem}

The key tools in the proof are the precise estimates for the $L^p$-means and $A^p_\nu$-norms of the Bergman kernel $B^\om_z$ obtained in \cite[Theorem~1]{PelRatKernels}. The special case of the said result is repeatedly used in the proof and stated for further reference as follows.

\begin{lettertheorem}\label{th:kernelstimate}
Let $0<p<\infty$ and $\om,\nu\in\DD$. Then the following assertions hold:
\begin{enumerate}
\item[\rm(i)]$\displaystyle M_p^p\left(r,B^\om_a\right)
=\frac{1}{2\pi}\int_0^{2\pi}|B^\om_z(re^{i\theta})|^p\,d\t\asymp
    \int_0^{|a|r}\frac{dt}{\widehat{\om}(t)^{p}(1-t)^{p}},\quad r,|a|\to1^-$;
\item[\rm(ii)]$\displaystyle
    \|B^\om_a\|^p_{A^p_\nu}
    \asymp \int_0^{|a|}\frac{\widehat{\nu}(t)}{\widehat{\om}(t)^{p}(1-t)^{p}}\,dt,\quad |a|\to1^-.
    $
\end{enumerate}
\end{lettertheorem}

The argument used to establish the one weight inequality \cite[Theorem~3]{PelRatKernels} for regular weights does not carry over as such to the two weight case. The proof of the sufficiency in Theorem~\ref{th:P+threeweights} is much more involved due to the presence of the second weight $\eta$.

The operators $P_\om$ and $P^+_\om$ are simultaneously bounded under a natural additional hypothesis. This is the content of the other main result of this study.

\begin{theorem}\label{th:Pthreeweights}
Let $1<p<\infty$, $\om,\nu\in\R$ and $\eta\in\DDD$ such that
    \begin{equation}\label{Eq:hypothesis:P}
    \int_0^r\frac{\eta(t)}{\widehat{\om}(t)^p}\,dt\lesssim\frac{\widehat{\eta}(r)}{\widehat{\om}(r)^p},\quad 0\le r<1.
    \end{equation}
Then the following statements are equivalent:
    \begin{itemize}
    \item[(i)] $P^+_\om: L^p_\nu\to L^p_\eta$ is bounded;
    \item[(ii)] $P_\om: L^p_\nu\to L^p_\eta$ is bounded;
    \item[(iii)] $\displaystyle N_p(\om,\nu,\eta)=\sup_{0<r<1}\frac{\widehat{\eta}(s)^\frac1p}{\widehat{\om}(s)}
    \left(\int_r^1\left(\frac{\om(s)}{\nu(s)^\frac{1}{p}}\right)^{p'}ds\right)^{\frac{1}{p'}}<\infty$;
    \item[(iv)] $M_p(\om,\nu,\eta)<\infty$.
    \end{itemize}
Moreover,
    $$
    \|P^+_\om\|_{L^p_\nu\to L^p_\eta}\asymp\|P_\om\|_{L^p_\nu\to L^p_\eta}\asymp M_p(\om,\nu,\eta)\asymp N_p(\om,\nu,\eta).
    $$
\end{theorem}

Although the conditions $N_p(\om,\nu,\eta)<\infty$ and $M_p(\om,\nu,\eta)<\infty$ are equivalent for many weights, for example standard weights have this property, the condition $N_p(\om,\nu,\eta)<\infty$ is of course weaker than $M_p(\om,\nu,\eta)<\infty$ in general.
Indeed, if we pick up an arbitrary $\om\in\R$, and define $\nu(s)=\om(s)\widehat{\om}(s)^{\frac{p}{p'}}\left(\log\frac{e}{1-s}\right)^{2\frac{p}{p'}}$
and $\eta(s)=\om(s)\widehat{\om}(s)^{\frac{p}{p'}}\left( \log\frac{e}{1-s}\right)^{s\frac{p}{p'}}$, then one can show by using Lemmas~\ref{Lemma:weights-in-D-hat-NEW}
and~\ref{Lemma:weights-in-R} below that $N_p(\om,\nu,\eta)<\infty$ but $M_p(\om,\nu,\eta)=\infty$.

It is readily seen that the methods used to prove Theorems~\ref{th:P+threeweights} and~\ref{th:Pthreeweights} carry over to the case $p=1$. In fact, the proof in this case turns out much more simple for obvious reasons. To be precise, one can show that in the case $p=1$ the operators $P_\om$ and $P^+_\om$ are simultaneously bounded, and the uniform boundedness of the quantity
    $$
    \frac{\om(r)}{\nu(r)}\int_0^1\frac{\eta(t)}{\widehat{\om}(tr)}\,dt
    $$
is the characterizing condition.

Throughout the paper $\frac{1}{p}+\frac{1}{p'}=1$ for $1<p<\infty$. Further, the letter $C=C(\cdot)$ will denote an
absolute constant whose value depends on the parameters indicated
in the parenthesis, and may change from one occurrence to another.
We will use the notation $a\lesssim b$ if there exists a constant
$C=C(\cdot)>0$ such that $a\le Cb$, and $a\gtrsim b$ is understood
in an analogous manner. In particular, if $a\lesssim b$ and
$a\gtrsim b$, then we will write $a\asymp b$.

\section{Proof of Theorem~\ref{th:P+threeweights}}

Throughout the proofs we will repeatedly use several basic properties of weights in the classes $\DD$ and $\Dd$, gathered in the following two lemmas. For a proof of the first lemma, see~\cite[Lemma~2.1]{PelSum14}; the second one can be proved by similar arguments.

\begin{letterlemma}\label{Lemma:weights-in-D-hat-NEW}
Let $\om$ be a radial weight. Then the following statements are equivalent:
\begin{itemize}
\item[\rm(i)] $\om\in\DD$;
\item[\rm(ii)] There exist constants $C=C(\om)>0$ and $\beta=\beta(\om)>0$ such that
    $$
    \widehat{\om}(r)\leq C\left(\frac{1-r}{1-t}\right)^{\b}\widehat{\om}(t),\quad 0\le r\le t<1;
    $$
\item[\rm(iii)] There exists a constant $C=C(\om)>0$ such that
    $$
    \om_x=\int_0^1 s^x\om(s)\,ds\le C\widehat{\om}\left(1-\frac1x\right), \quad x \in [1,\infty).
    $$
\end{itemize}
\end{letterlemma}

\begin{letterlemma}\label{Lemma:weights-in-R}
Let $\om$ be a radial weight. Then $\om\in\Dd$ if and only if there exist $C=C(\om)>0$ and $\g=\g(\om)>0$ such that
    \begin{equation*}
    \begin{split}
    \widehat{\om}(t)\le C\left(\frac{1-t}{1-r}\right)^{\g}\widehat{\om}(r),\quad 0\le r\le t<1.
    \end{split}
    \end{equation*}
\end{letterlemma}

Lemma~\ref{Lemma:weights-in-D-hat-NEW}(ii) shows that if $\om\in\DD$, then there exists $\beta=\beta(\om)>0$ such that $\frac{\widehat{\om}(r)}{(1-r)^\b}$ is essentially increasing on $[0,1)$. Similarly, by Lemma~\ref{Lemma:weights-in-R}, $\frac{\widehat{\om}(r)}{(1-r)^\gamma}$ is essentially decreasing on $[0,1)$ for $\gamma=\gamma(\om)>0$ sufficiently small if $\om\in\Dd$.

\subsection{Necessity}

In this section we prove that $M_p(\om,\nu,\eta)<\infty$ is a necessary condition for $P^+_\om: L^p_\nu\to L^p_\eta$ to be bounded under the hypotheses of Theorem~\ref{th:P+threeweights}, and establish the desired upper estimate for the operator norm. This is done under slightly weaker hypotheses than those of the theorem in the following result by using an appropriate family of test functions depending on the weights $\om$ and $\nu$.

\begin{proposition}\label{pr:necessity3d}
Let $1<p<\infty$, $\om\in\DD $ and $\nu,\eta$ radial weights. If
$P^+_\om: L^p_\nu\to L^p_\eta$ is bounded, then
    \begin{equation*}
    \sup_{0<r<1}\left(\int_0^r\left(J_\om(s)+1\right)^p\eta(s)\,sds\right)^{\frac{1}{p}}
    \left(\int_r^1\left(\frac{\om(s)}{\nu(s)^\frac{1}{p}}\right)^{p'}\,sds\right)^{\frac{1}{p'}}\lesssim\|P^+_\om\|_{L^p_\nu\to L^p_\eta}<\infty,
    \end{equation*}
where $J_\om(s)=\int_0^{s}\frac{dt}{\widehat{\om}(t)(1-t)}$ for all $0\le s<1$.
\end{proposition}

\begin{proof}
Assume that $P^+_\om: L^p_\nu\to L^p_\eta$ is bounded, that is,
    \begin{equation}\label{1}
    \int_\D\left|\int_\D f(\z)|B^\om_z(\z)|\om(\z)\,dA(\z)\right|^p\eta(z)\,dA(z)
    \le\|P^+_\om\|_{L^p_\nu\to L^p_\eta}^p\|f\|_{L^p_\nu}^p,\quad f\in L^p_\nu,
    \end{equation}
with $\|P^+_\om\|_{L^p_\nu\to L^p_\eta}<\infty$. If $\nu$ vanishes on a set $E\subset\D$ of positive measure, then by choosing $f=\chi_{E}$ the right side of \eqref{1} is zero. It follows that $\om$ vanishes (almost everywhere) on $E$ or else $\eta\equiv0$ (almost everywhere) on $\D$. The latter option being unacceptable as $\widehat{\eta}(r)>0$ for all $0\le r <1$, we deduce that $\om dA$ is absolutely continuous with respect to $\nu dA$. Therefore $\om/\nu$ is well defined almost everywhere.
Hence, for each $n\in\N$ and $0\le t<1$, the function $f_{n,t}=\min\left\{n,\left(\frac{\om}{\nu}\right)^\frac1{p-1}\right\}\chi_{\D\setminus D(0,t)}$ belongs to $L^p_\nu$. An application of Lemma~\ref{Lemma:weights-in-D-hat-NEW}(ii) to $\om\in\DD$ gives $J_\om(r)+1\asymp J_\om(r^2)+1$ for all $0\le r <1$. Therefore \eqref{1} and  Theorem~\ref{th:kernelstimate}(i) imply
    \begin{equation*}
    \begin{split}
    \|P^+_\om\|_{L^p_\nu\to L^p_\eta}^p\|f_{n,t}\|^p_{L^p_\nu}
    &\gtrsim\int_0^1 \left(\int_r^1f_{n,t}(s)\om(s)\,sds\right)^p\eta(r)\left(J_\om(r)+1\right)^pr\,dr\\
    &\ge\left(\int_0^t\eta(r)\left(J_\om(r)+1\right)^pr\,dr\right)\left(\int_t^1f_{n,t}(s)\om(s)\,sds\right)^p\\
    &\gtrsim\left(\int_0^t\eta(r)\left(J_\om(r)+1\right)^pr\,dr\right)\|f_{n,t}\|^{p^2}_{L^p_\nu},\quad n\in\N,\quad 0\le t<1,
    \end{split}
    \end{equation*}
because $f_{n,t}^p\nu\le f_{n,t}\om$ on $\D$. This together with the monotone convergence theorem shows that
    \begin{equation*}
    \begin{split}
    \|P^+_\om\|_{L^p_\nu\to L^p_\eta}
    &\gtrsim\left(\int_0^t\eta(r)\left(J_\om(r)+1\right)^pr \,dr\right)^\frac1p
    \left(\int_t^1\left(\frac{\om(s)}{\nu(s)^\frac1p}\right)^{p'}\,sds\right)^{\frac1{p'}},\quad 0\le t<1,
    \end{split}
    \end{equation*}
and the proposition is proved.
\end{proof}

If $\om\in\DD$, then by Lemma~\ref{Lemma:weights-in-D-hat-NEW}(ii) there exists $\b=\b(\om)>0$ such that $J_\om(r)\gtrsim\widehat{\om}(r)^{-1}(1-(1-r)^\b)$ for all $0\le r <1$. Therefore, under the hypotheses of Theorem~\ref{th:P+threeweights}, we have $\|P^+_\om\|_{L^p_\nu\to L^p_\eta}\gtrsim M_p(\om,\nu,\eta)$, and thus the necessity part is proved.

\subsection{Sufficiency}

The proof of the sufficiency of $M_p(\om,\nu,\eta)<\infty$ for $P^+_\om: L^p_\nu\to L^p_\eta$ to be bounded is more involved than that of the necessity. We begin with the following technical lemma.

\begin{lemma}\label{le:nec3}
Let $1<p<\infty$ and $\om\in\DDD$ and $\nu\in\Dd$ such that
    $$
    \int_0^1
    \left(\frac{\widehat{\om}(s)}{\widehat{\nu}(s)^\frac{1}{p}}\right)^{p'}\frac{ds}{1-s}<\infty.
    $$
Then
    \begin{equation*}
    \int_0^r \left(\int_t^1
    \left(\frac{\widehat{\om}(s)}{\widehat{\nu}(s)^\frac{1}{p}}\right)^{p'}\frac{ds}{1-s}
    \right)^{\frac{1}{p'}}\frac{dt}{\widehat{\om}(t)(1-t)}\lesssim
    \left(
    \int_r^1\left(\frac{\widehat{\om}(s)}{\widehat{\nu}(s)^\frac{1}{p}}\right)^{p'}\frac{ds}{1-s}\right)^{\frac{1}{p'}}
    \frac{1}{\widehat{\om}(r)},\quad 0\le r<1.
    \end{equation*}
\end{lemma}

\begin{proof}
Let $\alpha=\alpha(\om,\nu,p)\in(0,1)$ to be appropriately fixed later. Then H\"older's inequality and Lemma~\ref{Lemma:weights-in-R} yield
    \begin{equation}
    \begin{split}\label{eq:nec1}
    &\int_0^r\left(\int_t^1
    \left(\frac{\widehat{\om}(s)}{\widehat{\nu}(s)^\frac{1}{p}}\right)^{p'}\frac{ds}{1-s}
    \right)^{\frac{1}{p'}}\frac{dt}{\widehat{\om}(t)(1-t)}\\
    &\le\left(\int_0^r\left(\int_t^1
    \left(\frac{\widehat{\om}(s)}{\widehat{\nu}(s)^{\frac{1}{p}}}
    \right)^{p'}\frac{ds}{1-s}\right)
    \frac{dt}{\widehat{\om}(t)^{p'\alpha}(1-t)}\right)^{\frac{1}{p'}}
    \left(\int_0^r\frac{dt}{\widehat{\om}(t)^{p(1-\alpha)}(1-t)}\right)^{\frac{1}{p}}\\
    &\lesssim\left(\int_0^r\left(\int_t^1
    \left(\frac{\widehat{\om}(s)}{\widehat{\nu}(s)^{\frac{1}{p}}}
    \right)^{p'}\frac{ds}{1-s}\right)
    \frac{dt}{\widehat{\om}(t)^{p'\alpha}(1-t)}
    \right)^{\frac{1}{p'}}\frac{1}{\widehat{\om}(r)^{(1-\alpha)}},\quad 0\le r<1,
    \end{split}
    \end{equation}
where, by Fubini's theorem and Lemma~\ref{Lemma:weights-in-R},
    \begin{equation}
    \begin{split}\label{eq:nec2}
    &\int_0^r\left(\int_t^1
    \left(\frac{\widehat{\om}(s)}{\widehat{\nu}(s)^{\frac{1}{p}}}
    \right)^{p'}\frac{ds}{1-s}\right)
    \frac{dt}{\widehat{\om}(t)^{p'\alpha}(1-t)}\\
    &=\int_0^r\left(\frac{\widehat{\om}(s)}{\widehat{\nu}(s)^{\frac{1}{p}}}
    \right)^{p'}\left(\int_0^s
    \frac{dt}{\widehat{\om}(t)^{p'\alpha}(1-t)}\right)\frac{ds}{1-s}\\
    &\quad+\left(\int_r^1\left(\frac{\widehat{\om}(s)}{\widehat{\nu}(s)^{\frac{1}{p}}}\right)^{p'}\frac{ds}{1-s}\right)
    \left(\int_0^r\frac{dt}{\widehat{\om}(t)^{p'\alpha}(1-t)}\right)\\
    &\lesssim\int_0^r
    \left(\frac{\widehat{\om}(s)^{1-\alpha}}{\widehat{\nu}(s)^{\frac{1}{p}}}\right)^{p'}\frac{ds}{1-s}
    +\left(\int_r^1
    \left(\frac{\widehat{\om}(s)}{\widehat{\nu}(s)^{\frac{1}{p}}}
    \right)^{p'}\frac{ds}{1-s}\right)
    \frac{1}{\widehat{\om}(r)^{p'\alpha}},\quad 0\le r<1.
    \end{split}
    \end{equation}
The latter term is of the desired form. To deal with the first term, observe first that by Lemma~\ref{Lemma:weights-in-D-hat-NEW}(ii) there exists a constant $\b=\b(\om)>0$ such that $\frac{\widehat{\om}(r)}{(1-r)^{\b}}$ is essentially increasing on $[0,1)$. Further, for each sufficiently small $\gamma=\gamma(\nu)>0$ the function $\frac{\widehat{\nu}(r)}{(1-r)^{\gamma}}$ is essentially decreasing on $[0,1)$ by Lemma~\ref{Lemma:weights-in-R}. Pick up such a $\gamma$ from the interval $(0,p\b)$, and fix $\alpha\in(1-\frac{\gamma}{p\b},1)$. Then
    \begin{equation*}
    \begin{split}
    \int_0^r\left(\frac{\widehat{\om}(s)^{1-\alpha}}{\widehat{\nu}(s)^{\frac{1}{p}}}
    \right)^{p'}\frac{ds}{1-s}
    &\lesssim\left(\frac{\widehat{\om}(r)^{1-\alpha}}{(1-r)^{\b(1-\alpha)}}\right)^{p'}
    \int_0^r\left(\frac{(1-s)^{\b(1-\alpha)}}{\widehat{\nu}(s)^{\frac{1}{p}}}\right)^{p'}\frac{ds}{1-s}\\
    &\lesssim\left(\frac{\widehat{\om}(r)^{1-\alpha}}{(1-r)^{\b(1-\alpha)}}\right)^{p'}
    \left(\frac{(1-r)^{\frac{\g}{p}}}{\widehat{\nu}(r)^{\frac{1}{p}}}\right)^{p'}
    \int_0^r\frac{ds}{(1-s)^{1+p'\left(\frac{\g}{p}-\b(1-\alpha)\right)}}\\
    &\lesssim\left(\frac{\widehat{\om}(r)^{1-\alpha}}{\widehat{\nu}(r)^{\frac{1}{p}}}\right)^{p'}
    \lesssim\left(\int_r^{\frac{1+r}{2}}\left(\frac{\widehat{\om}(s)}{\widehat{\nu}(s)^{\frac{1}{p}}}\right)^{p'}\frac{ds}{1-s}\right)
    \frac{1}{\widehat{\om}(r)^{p'\alpha}}\\
    &\lesssim\left(\int_r^1\left(\frac{\widehat{\om}(s)}{\widehat{\nu}(s)^{\frac{1}{p}}}\right)^{p'}\frac{ds}{1-s}\right) \frac{1}{\widehat{\om}(r)^{p'\alpha}},\quad 0\le r<1,
    \end{split}
    \end{equation*}
which together with \eqref{eq:nec2} gives
    \begin{equation*}
    \begin{split}
    \int_0^r \left( \int_t^1
    \left(\frac{\widehat{\om}(s)}{\widehat{\nu}(s)^{\frac{1}{p}}}\right)^{p'}\frac{ds}{1-s}\right)
    \frac{dt}{\widehat{\om}(t)^{p'\alpha}(1-t)}
    \lesssim\left(\int_r^1\left(\frac{\widehat{\om}(s)}{\widehat{\nu}(s)^{\frac{1}{p}}}
    \right)^{p'}\frac{ds}{1-s}\right)\frac{1}{\widehat{\om}(r)^{p'\alpha}},\quad 0\le r<1.
    \end{split}
    \end{equation*}
Finally, by combining the above inequality with \eqref{eq:nec1} we
obtain the claim.
\end{proof}

We are now ready to prove the sufficiency part of Theorem~\ref{th:P+threeweights}. To do this, assume $M_p(\om,\nu,\eta)<\infty$, and observe that then the function $h(z)=\nu^{1/p}(z)\left(\int_{|z|}^1\left(\frac{\om(s)}{\nu(s)}\right)^{p'}\nu(s)\,ds\right)^{\frac{1}{pp'}}$ is well defined for all $z\in\D$. Hence an integration shows that
    \begin{equation}
    \begin{split}\label{eq:mu1}
    \int_{t}^1 \left(\frac{\om(s)}{h(s)}\right)^{p'}ds
    =p'\left(\int_t^1\left(\frac{\om(s)}{\nu(s)}\right)^{p'}\nu(s)\,ds\right)^{\frac{1}{p'}},\quad 0\le t<1.
    \end{split}
    \end{equation}
H\"older's inequality yields
    \begin{equation}
    \begin{split}\label{eq:mu2}
    \|P^+_\om(f)\|^p_{L^p_\eta}
    &\le\int_{\D}\left(\int_{\D}|f(\z)|^ph(\z)^p|B^\om_z(\z)|\,dA(\z)\right)\\
    &\quad\cdot\left(\int_{\D}|B^\om_z(\z)|\left(\frac{\om(\z)}{h(\z)}\right)^{p'}dA(\z)\right)^{p/p'}\eta(z)\,dA(z),
    \end{split}
    \end{equation}
where, by Theorem~\ref{th:kernelstimate}(i), Fubini's theorem and \eqref{eq:mu1},
    \begin{equation*}
    \begin{split}
    \int_{\D}|B^\om_z(\z)|\left(\frac{\om(\z)}{h(\z)}\right)^{p'}dA(\z)
    &\lesssim \int_0^{1}\left(\frac{\om(s)}{h(s)}\right)^{p'}\left(\int_0^{s|z|}\frac{dt}{\widehat{\om}(t)(1-t)}\right)\,ds+1\\
    &=\int_0^{|z|}\left(\int_{t/|z|}^1\left(\frac{\om(s)}{h(s)}\right)^{p'}ds\right)\frac{dt}{\widehat{\om}(t)(1-t)}+1\\
    &\le\int_0^{|z|}\left(\int_{t}^1\left(\frac{\om(s)}{h(s)}\right)^{p'}ds\right)\frac{dt}{\widehat{\om}(t)(1-t)}+1\\
    &=p'\int_0^{|z|}\left(\int_t^1\left(\frac{\om(s)}{\nu(s)}\right)^{p'}\nu(s)\,ds\right)^{\frac{1}{p'}}\frac{dt}{\widehat{\om}(t)(1-t)}+1.
    \end{split}
    \end{equation*}
This together with \eqref{eq:mu2}, Fubini's theorem and another
application of Theorem~\ref{th:kernelstimate}(i) gives
    \begin{equation}
    \begin{split}\label{eq:mu4}
    \|P^+_\om(f)\|^p_{L^p_\eta}
    &\lesssim
    \int_{\D}\left(\int_{\D}|f(\z)|^ph(\z)^p|B^\om_z(\z)|\,dA(\z)\right)\\
    &\quad \cdot\left(\int_0^{|z|}\left(\int_t^1\left(\frac{\om(s)}{\nu(s)}\right)^{p'}\nu(s)\,ds\right)^{\frac{1}{p'}}\frac{dt}{\widehat{\om}(t)(1-t)}\right)^{p/p'}
    \eta(z)\,dA(z)+I_1(f)\\
    &\lesssim\int_{\D}|f(\z)|^ph(\z)^{p}
    \Bigg[\int_0^1
    \left(\int_0^{r}\left(\int_t^1\left(\frac{\om(s)}{\nu(s)}\right)^{p'}\nu(s)\,ds\right)^{\frac{1}{p'}}
    \frac{dt}{\widehat{\om}(t)(1-t)}\right)^{p/p'}\\
    &\quad\cdot\left(\int_0^{r|\z|}\frac{dx}{\widehat{\om}(x)(1-x)}\right)\eta(r)r\,dr\Bigg]\,dA(\z)+I_1(f)+I_2(f),
    \end{split}
    \end{equation}
where
    \begin{equation*}
    \begin{split}
    I_1(f)
    &=\int_\D\left(\int_\D|f(\z)|^ph(\z)^p|B^\om_z(\z)|\,dA(\z)\right)\eta(z)\,dA(z)\\
    \end{split}
    \end{equation*}
and
    \begin{equation*}
    \begin{split}
    I_2(f)&=\int_{\D}|f(\z)|^ph(\z)^{p}
    \left(\int_0^1
    \left(\int_0^{r}\left(\int_t^1\left(\frac{\om(s)}{\nu(s)}\right)^{p'}\nu(s)\,ds\right)^{\frac{1}{p'}}
    \frac{dt}{\widehat{\om}(t)(1-t)}\right)^{p/p'}\eta(r)r\,dr\right)\,dA(\z).
    \end{split}
    \end{equation*}
By Fubini's theorem, Theorem~\ref{th:kernelstimate}(ii) H\"older's inequality,
    \begin{equation*}
    \begin{split}
    I_1(f)
    &=\int_\D\left(\int_\D|f(\z)|^ph(\z)^p|B^\om_z(\z)|\,dA(\z)\right)\eta(z)\,dA(z)\\
    &\lesssim\int_\D|f(\z)|^p\nu(\z)\left(\int_{|\z|}^1\left(\frac{\om(s)}{\nu(s)}\right)^{p'}\nu(s)\,ds\right)^\frac1{p'}
    \left(\int_0^{|\z|}\frac{\widehat{\eta}(t)}{\widehat{\om}(t)(1-t)}\,dt\right)+\|f\|_{L^p_\nu}^p\\
    &\lesssim M_p(\om,\nu,\eta) \int_\D|f(\z)|^p\nu(\z)\left(\int_{|\z|}^1\left(\frac{\om(s)}{\nu(s)}\right)^{p'}\nu(s)\,ds\right)^\frac1{p'}
    \left(\int_0^{|\z|}\frac{\widehat{\eta}(t)}{\widehat{\om}(t)^p(1-t)}\,dt\right)^\frac1p+\|f\|_{L^p_\nu}^p\\
    &\lesssim M_p(\om,\nu,\eta)\|f\|_{L^p_\nu}^p
    \end{split}
    \end{equation*}
because $\widehat{\eta}(r)/(1-|\z|)$ is a weight by the hypothesis $\eta\in\Dd$ and Lemma~\ref{Lemma:weights-in-R}. Fubini's theorem and Lemma~\ref{Lemma:weights-in-R} for $\om\in\Dd$ give
    $$
    \int_0^{|\z|}\frac{\widehat{\eta}(t)}{\widehat{\om}(t)^p(1-t)}\,dt
    \lesssim\frac{\widehat{\eta}(\z)}{\widehat{\om}(\z)^p}+\int_0^{|\z|}\frac{\eta(t)}{\widehat{\om}(t)^p}\,dt,\quad \z\in\D.
    $$
But since $\eta\in\Dd$ by the hypothesis, there exists a constant $K=K(\eta)>1$ such that $\eta$ satisfies \eqref{K}. Hence, by using Lemma~\ref{Lemma:weights-in-D-hat-NEW}(ii) for $\om\in\DD$, we deduce
    $$
    \int_0^r\frac{\eta(s)}{\widehat{\om}(s)^p}\,ds\ge \int_{1-K(1-r)}^r\frac{\eta(s)}{\widehat{\om}(s)^p}\,ds
    \ge\frac{\int_{1-K(1-r)}^r \eta(s)\,ds}{\widehat{\om}(1-K(1-r))^p}\gtrsim \frac{\widehat{\eta}(r)}{\widehat{\om}(r)^p},\quad r\ge 1-K^{-1},
    $$
and hence
    \begin{equation}\label{eq:necessity3w1}
    \sup_{0<r<1}
    \frac{\widehat{\eta}(r)^{\frac{1}{p}}}{\widehat{\om}(r)}
    \left(\int_r^1\left(\frac{\om(s)}{\nu(s)^\frac{1}{p}}\right)^{p'}\,ds\right)^{\frac{1}{p'}}\lesssim M_p(\om,\nu,\eta)<\infty.
    \end{equation}
It follows that $I_1(f)\lesssim \|f\|_{L^p_\nu}^p$. To deal with the remaining terms, we split the integral over $(0,1)$ in \eqref{eq:mu4} into two parts at $|\z|$. On one hand, since $\om,\eta\in\Dd$, Lemma~\ref{Lemma:weights-in-R} and \eqref{eq:necessity3w1} yield
    \begin{equation}
    \begin{split}\label{eq:mu5}
    &\int_{|\z|}^1
    \left(\int_0^{r}\left(\int_t^1\left(\frac{\om(s)}{\nu(s)}\right)^{p'}\nu(s)\,ds\right)^{\frac{1}{p'}}\frac{dt}{\widehat{\om}(t)(1-t)}\right)^{p/p'}
    \left(\int_0^{r|\z|}\frac{dx}{\widehat{\om}(x)(1-x)}\right)\eta(r)\,dr\\
    &\lesssim\frac{1}{\widehat{\om}(\z)}
    \int_{|\z|}^1\left(\int_0^{r}
    \left(\int_t^1\left(\frac{\om(s)}{\nu(s)}\right)^{p'}\nu(s)\,ds\right)^{\frac{1}{p'}}\frac{dt}{\widehat{\om}(t)(1-t)}\right)^{p/p'}\eta(r)\,dr\\
    &\lesssim M^{\frac{p}{p'}}_p(\om,\nu,\eta)\frac{1}{\widehat{\om}(\z)}
    \int_{|\z|}^1\left(\int_0^{r}\frac{dt}{\widehat{\eta}(t)^{\frac{1}{p}}(1-t)}\right)^{p/p'}\eta(r)\,dr\\
    &\asymp M^{\frac{p}{p'}}_p(\om,\nu,\eta)\frac{1}{\widehat{\om}(\z)}\int_{|\z|}^1\eta(r)\widehat{\eta}(r)^{-\frac{1}{p'}}dr
    = M^{\frac{p}{p'}}_p(\om,\nu,\eta)p\frac{\widehat{\eta}(\z)^{\frac{1}{p}}}{\widehat{\om}(\z)},\quad \z\in\D.
    \end{split}
    \end{equation}
Therefore, by using \eqref{eq:necessity3w1} again we deduce
    \begin{equation*}
    \begin{split}
    &h^p(\z)\int_{|\z|}^1
    \left(\int_0^{r}\left(\int_t^1\left(\frac{\om(s)}{\nu(s)}\right)^{p'}\nu(s)\,ds\right)^{\frac{1}{p'}}\frac{dt}{\widehat{\om}(t)(1-t)}\right)^{p/p'}
    \left(\int_0^{r|\z|}\frac{dx}{\widehat{\om}(x)(1-x)}\right)\eta(r)r\,dr\\
    &\lesssim M^{\frac{p}{p'}}_p(\om,\nu,\eta)\nu(\z)
    \left(\int_{|\z|}^1\left(\frac{\om(s)}{\nu(s)}\right)^{p'}\nu(s)\,ds\right)^{\frac{1}{p'}}
    \frac{\widehat{\eta}(\z)^{\frac{1}{p}}}{\widehat{\om}(\z)}\lesssim M^{p}_p(\om,\nu,\eta)\nu(\z),\quad \z\in\D.
    \end{split}
    \end{equation*}
On the other hand, since $\om,\nu\in\R$, Lemma~\ref{le:nec3}, Lemma~\ref{Lemma:weights-in-R} and the hypothesis $M_p(\om,\nu,\eta)<\infty$ yield
    \begin{equation*}
    \begin{split}
    &\int_0^{|\z|}
    \left(\int_0^{r}\left(\int_t^1\left(\frac{\om(s)}{\nu(s)}\right)^{p'}\nu(s)\,ds\right)^{\frac{1}{p'}}\frac{dt}{\widehat{\om}(t)(1-t)}\right)^{p/p'}
    \left(\int_0^{r|\z|}\frac{dx}{\widehat{\om}(x)(1-x)}\right)\eta(r)r\,dr\\
    &\lesssim\int_0^{|\z|}
    \left(\int_r^1\left(\frac{\om(s)}{\nu(s)}\right)^{p'}\nu(s)\,ds\right)^{\frac{p}{(p')^2}}\frac{1}{\widehat{\om}(r)^{p-1}}
    \left(\int_0^{r}\frac{dx}{\widehat{\om}(x)(1-x)}\right)\eta(r)\,dr\\
    &\lesssim M^{\frac{p}{p'}}_p(\om,\nu,\eta)\int_0^{|\z|}
    \frac{\eta(r)}{\widehat{\om}(r)^{p}\left(\int_0^r
    \frac{\eta(s)}{\widehat{\om}(s)^p}\,ds\right)^{\frac{1}{p'}}}dr\asymp
    M^{\frac{p}{p'}}_p(\om,\nu,\eta) \left( \int_0^{|\z|}
    \frac{\eta(s)}{\widehat{\om}(s)^p}\,ds\right)^{\frac{1}{p}},\quad \z\in\D.
    \end{split}
    \end{equation*}
This together with the hypothesis $M_p(\om,\nu,\eta)<\infty$ gives
    \begin{equation*}
    \begin{split}
    &h^p(\z)\int_0^{|\z|}
    \left(\int_0^{r}\left(\int_t^1\left(\frac{\om(s)}{\nu(s)}\right)^{p'}\nu(t)\,dt\right)^{\frac{1}{p'}}\frac{dt}{\widehat{\om}(t)(1-t)}\right)^{p/p'}
    \left(\int_0^{r|\z|}\frac{dx}{\widehat{\om}(x)(1-x)}\right)\eta(r)r\,dr\\
    &\lesssim M^{\frac{p}{p'}}_p(\om,\nu,\eta)\nu(\z)
    \left(\int_{|\z|}^1\left(\frac{\om(s)}{\nu(s)}\right)^{p'}\nu(s)\,ds\right)^{\frac{1}{p'}}
    \left(\int_0^{|\z|}
    \frac{\eta(s)}{\widehat{\om}(s)^p}\,ds\right)^{\frac{1}{p}}
    \lesssim M^{p}_p(\om,\nu,\eta)\nu(\z),\quad \z\in\D.
    \end{split}
    \end{equation*}
Consequently, by combining the previous estimates, we deduce that the third last term in \eqref{eq:mu4} is bounded by a constant times $M^{p}_p(\om,\nu,\eta)\|f\|_{L^p_\nu}^p$. Essentially the same reasoning shows that $I_2(f)$ obeys the same upper bound. Thus, by putting everything together, we see that $\|P^+_\om(f)\|_{L^p_\eta}\lesssim M_p(\om,\nu,\eta)\|f\|_{L^p_\nu}$ as claimed.

\section{Proof of Theorem~\ref{th:Pthreeweights}}

Trivially (i) implies (ii), and $\|P_\om^+\|_{L^p_\nu\to L^p_\eta}\lesssim M_p(\om,\nu,\eta)$ by Theorem~\ref{th:P+threeweights}. Since the hypothesis \eqref{Eq:hypothesis:P} guarantees $M_p(\om,\nu,\eta)\lesssim N_p(\om,\nu,\eta)$, it suffices to show that $N_p(\om,\nu,\eta)\lesssim \|P_\om\|_{L^p_\nu\to L^p_\eta}$. The following result gives this implication and completes the proof of Theorem~\ref{th:Pthreeweights}.

\begin{proposition}
Let $1<p<\infty$, $\om\in\DD$, $\eta\in\DDD$ and $\nu$ a radial weight. If $P_\om : L^p_\nu \to L^p_\eta$ is bounded, then
    $$
    \sup_{0<r<1}\frac{\widehat{\eta}\left(r\right)^{\frac1p}}{\widehat{\om}\left(r\right)}\left(
    \int_{r}^1\left(\frac{\om(t)}{\nu(t)}\right)^{p'}\nu(t)\,tdt\right)^\frac1{p'}\lesssim\|P_\om\|_{L^p_\nu\to L^p_\eta}.
    $$
\end{proposition}

\begin{proof}
The adjoint of $P_\om$ is defined by
    \begin{equation*}
    \langle P_\om(f),g\rangle_{L^2_\eta}
    =\langle f,P^\star_\om(g)\rangle_{L^2_\nu}, \quad f\in L^p_\nu,\quad  g\in  L^{p'}_\eta.
    \end{equation*}
Now \cite[Theorem~1(i)]{PelRatKernels} and Lemma~\ref{Lemma:weights-in-R}, applied to $\eta\in\Dd$, yield
    \begin{equation}\label{3}
    \begin{split}
    &\int_{\D}\left(\int_\D|B^\om_\z(z)|\om(\z)\,dA(\z)\right)\eta(z)\,dA(z)
    \lesssim\int_0^1\eta(r)\log\frac{e}{1-r}\,dr
    =\int_0^1\frac{\widehat{\eta}(r)}{1-r}\,dr<\infty.
    \end{split}
    \end{equation}
If $f$ and $g$ are bounded functions, then \eqref{3} shows that we may apply Fubini's theorem to deduce
    \begin{equation*}
    \begin{split}
    \langle P_\om(f),g\rangle_{L^2_\eta}
    &=\int_{\D} P_\om(f)(z)\overline{g(z)}\eta(z)\,dA(z)\\
    &=\int_{\D}\left(\int_\D f(\z)B^\om_\z(z)\om(\z)\,dA(\z)\right)\overline{g(z)}\eta(z)\,dA(z)\\
    &=\int_{\D}f(\z)\left(\int_\D\overline{g(z)}B^\om_\z(z)\eta(z)\,dA(z)\right)\om(\z)\,dA(\z)\\
    &=\int_{\D}f(\z)\overline{\left(\frac{\om(\z)}{\nu(\z)}\int_\D g(z)B^\om_z(\z)\eta(z)\,dA(z) \right)} \nu(\z)\,dA(\z)
    =\langle f,P^\star_\om(g)\rangle_{L^2_\nu}.
    \end{split}
    \end{equation*}
Since the simple functions are dense in $L^p_\sigma$ for each $1\le p<\infty$ and radial $\sigma$, this shows that
    \begin{equation*}
    P^\star_\om(g)(\z)=\frac{\om(\z)}{\nu(\z)}\int_{\D}g(z)B^\om_z(\z)\eta(z)\,dA(z),\quad \z\in\D,\quad g\in L^{p'}_\eta.
    \end{equation*}
The adjoint operator $P^\star_\om:L^{p'}_\eta\to L^{p'}_\nu$ is bounded by the hypothesis, and $\|P^\star_\om\|_{L^{p'}_\eta \to L^{p'}_\nu} = \|P_\om\|_{L^p_\nu \to L^p_\eta}$. Thus
    \begin{equation}\label{eq:adjointbdd}
    \int_\D \left(\frac{\om(\zeta)}{\nu(\zeta)}\right)^{p'} \left|\int_\D g(z) B^\om_z(\zeta) \eta(z)\,dA(z)\right|^{p'} \nu(\zeta)\,dA(\zeta)
    \le\|P_\om\|^{p'}_{L^p_\nu \to L^p_\eta}\|g\|_{L^{p'}_\eta}^{p'},\quad g \in L^{p'}_\eta.
    \end{equation}
By considering the standard orthonormal basis $\{z^j/\sqrt{2\om_{2j+1}}\}$, $j\in\N\cup\{0\}$, of the Hilbert space $A^2_\om$, one deduces
    $$
    B^\om_z(\zeta)=\sum_{n=0}^\infty\frac{(\zeta\overline{z})^n}{2\om_{2n+1}},\quad z,\zeta\in\D.
    $$
By testing \eqref{eq:adjointbdd} with monomials $g_n(\z)=\z^n$ we obtain
    \begin{equation}\label{4}
    \begin{split}
    \left(\frac{\eta_{2n+1}}{\om_{2n+1}}\right)^{p'}\int_\D|\z|^{np'}\left(\frac{\om(\z)}{\nu(\z)}\right)^{p'}\nu(\z)\,dA(\z)
    &\asymp\|P_\om^\star(g_n)\|_{L^{p'}_\eta}^{p'}
    \le\|P_\om\|^{p'}_{L^p_\nu \to L^p_\eta}\|g_n\|_{L^{p'}_\eta}^{p'}\\
    &=2\|P_\om\|^{p'}_{L^p_\nu \to L^p_\eta}\eta_{np'+1},\quad n\in\N\cup\{0\},
    \end{split}
    \end{equation}
from which Lemma~\ref{Lemma:weights-in-D-hat-NEW}(ii)(iii), applied to $\om,\eta\in\DD$, yields
    \begin{equation*}
    \begin{split}
    \|P_\om\|^{p'}_{L^p_\nu \to L^p_\eta}
    &\gtrsim\sup_{n\in\N}\left(\frac{\eta^{p'}_{2n+1}}{\om^{p'}_{2n+1}\eta_{np'+1}}\int_0^1t^{np'}\left(\frac{\om(t)}{\nu(t)}\right)^{p'}\nu(t)t\,dt\right)\\
    &\asymp\sup_{n\in\N}\left(\frac{\widehat{\eta}\left(1-\frac1n\right)^{p'-1}}{\widehat{\om}\left(1-\frac1n\right)^{p'}}
    \int_0^1t^{np'}\left(\frac{\om(t)}{\nu(t)}\right)^{p'}\nu(t)\,tdt\right)\\
    &\gtrsim\sup_{n\in\N\setminus\{1\}}\left(\frac{\widehat{\eta}\left(1-\frac1n\right)^{p'-1}}{\widehat{\om}\left(1-\frac1n\right)^{p'}}
    \int_{1-\frac1n}^1\left(\frac{\om(t)}{\nu(t)}\right)^{p'}\nu(t)\,tdt\right).
    \end{split}
    \end{equation*}
Let $\frac12\le r<1$ and fix $n\in\N$ such that $1-\frac1n\le r<1-\frac{1}{n+1}$. By applying Lemma~\ref{Lemma:weights-in-D-hat-NEW}(ii) again we finally deduce
    \begin{equation*}
    \begin{split}
    \|P_\om\|_{L^p_\nu \to L^p_\eta}
    \gtrsim\frac{\widehat{\eta}\left(r\right)^{\frac1p}}{\widehat{\om}\left(r\right)}\left(
    \int_{r}^1\left(\frac{\om(t)}{\nu(t)}\right)^{p'}\nu(t)\,tdt\right)^\frac1{p'},\quad \frac12\le r<1.
    \end{split}
    \end{equation*}
The assertion follows from this inequality because the last integral converges for $r=0$ by \eqref{4} with $n=0$.
\end{proof}

\end{document}